\documentclass[10pt]{amsart}

\usepackage[margin=3cm]{geometry}

\usepackage{amsthm, amssymb, amsmath, hyperref, dsfont, enumitem, tcolorbox, esint, mathtools}
\usepackage{todonotes}  
\usepackage{orcidlink}

\usepackage{tikz-cd,graphicx}
\usepackage[all]{xy}
\graphicspath{ {./images/} }

\mathtoolsset{showonlyrefs,showmanualtags}

\newcommand{\Mdef}[2]{\newcommand{#1}{\relax \ifmmode #2 \else $#2$\fi}}
\newcommand{\oline}{\overline}

\newtheorem{theorem}{Theorem}[section]
\newtheorem{lemma}[theorem]{Lemma}
\newtheorem{prop}[theorem]{Proposition}
\newtheorem{definition}[theorem]{Definition}

\newtheorem{remark}[theorem]{Remark}

\numberwithin{equation}{section}
\usepackage{contour}
\contourlength{0.8pt}

\DeclareMathOperator{\rank}{rank}
\DeclareMathOperator{\depr}{depr}

\Mdef{\F}{\mathbb{F}}
\Mdef{\N}{\mathbb{N}}
\Mdef{\R}{\mathbb{R}}
\Mdef{\Z}{\mathbb{Z}}
\Mdef{\Q}{\mathbb{Q}}
\Mdef{\A}{\mathbb{A}}
\Mdef{\C}{\mathbb{C}}
\Mdef{\One}{\mathds{1}}
\Mdef{\RP}{\mathbb{R}\text{P}}

\newcommand{\ve}{\varepsilon}

\title[Canonical frames for bracket generating rank $2$ distributions which are not Goursat] {Canonical frames for bracket generating rank $2$ distributions which are not Goursat}
\date{\today}

\thanks{ I.\ Zelenko is partly supported by NSF grant DMS 2105528 and Simons Foundation Collaboration Grant for Mathematicians 524213.}

\author{Nicklas Day}
\address{Nicklas Day\\
	Department of Mathematics\\
	Texas A\&M University\\
	College Station\\
	Texas \ 77843\\
	USA}
\email{ncday@tamu.edu}
\urladdr{\url{https://sites.google.com/tamu.edu/nicklasday/home}}

\author{Igor Zelenko \orcidlink{0000-0001-7900-2567}} 
\address{Igor Zelenko\\
         Department of Mathematics\\
         Texas A\&M University\\
         College Station\\
         Texas \ 77843\\
         USA}
\email{zelenkotamu@tamu.edu}
\urladdr{\url{http://www.math.tamu.edu/~zelenko}}

\begin{document}
\subjclass[2020]{58A30, 58A17, 34H05, 37J60}
\keywords{distributions (subbubdles of tangent bundles), canonical frames, symmetries, Goursat distributions, abnormal extremals.}
\begin{abstract}
We complete a uniform construction of canonical absolute parallelism for bracket generating rank $2$ distributions with $5$-dimensional cube on $n$-dimensional manifold with $n\geq 5$ by showing that the condition of maximality of class that was assumed previously by Doubrov-Zelenko for such a construction holds automatically at generic points. This also gives analogous constructions in the case when the cube is not $5$-dimensional but the distribution is not Goursat through the procedure of iterative Cartan deprolongation. This together with the classical theory of Goursat distributions covers in principle the 
local geometry of all bracket generating rank 2 distributions in a neighborhood of generic points. As a byproduct, for any $n\geq 5$ we describe the maximally symmetric germs among bracket generating rank $2$ distributions with $5$-dimensional cube, as well as among those which reduce to such a distribution under a fixed number of Cartan deprolongations. Another consequence of our results on maximality of class is for optimal control problems with constraint given by a rank $2$ distribution with $5$-dimensional cube: it implies that for a generic point $q_0$ of $M$, there are plenty abnormal extremal trajectories of corank $1$ (which is the minimal possible corank) starting at $q_0$. The set of such points contains all points where the distribution is equiregular.

\end{abstract}
\maketitle

\section{Main results}
This paper is devoted to the proof of the existence of a canonical absolute parallelism and to the characterization of maximally symmetric models for bracket generating rank $2$ distributions that satisfy a natural condition of maximal growth for iterated Lie brackets of length at most $3$ of their sections with further consequences to rank $2$ distributions which are not Goursat.

A rank $l$ vector distribution $D$ on an $n$-dimensional
manifold $M$ or an $(l,n)$-distribution (where $l<n$) is a subbundle of the tangent bundle $TM$ with $l$-dimensional fibers. The \emph{weak derived flag at $p\in M$} of the distribution $D$ is the flag $\{D^{i}\}_{i=1}^\infty$ defined by
\begin{equation}
\label{week_derived}
    D^{1}(q) = D(q)\quad \text{and}\quad D^{i}(q) = D^{i-1}(q)+[D,D^{i-1}](q)\ \text{for } i>1.
\end{equation}
The space $D^i(q)$ is called the \emph{$i$th power} of the distribution $D$ at point $q$. In particular, $D^2(q)$ (respectively, $D^3(q)$) is called the square (respectively, the cube) of the distribution $D$ at $q$. A distribution $D$ is called \emph{bracket generating} if for every $q$, $D^k(q)$ coincides with the whole tangent space $T_q M$ for sufficiently large $k$. The tuple $(\dim D(q). \dim D^2(q), \ldots \dim D^i(q), \cdots)$ is called the \emph{small growth vector of the distribution $D$ at the point $q$.} The distribution $D$ is called \emph{equiregular} at a point $q_0$ if there exists a neighborhood $U$ of $q_0$ such that the small growth vector of $D$ is the same for all $q\in U$. If $D$ is bracket generating, the set of points at which $D$ is equiregular, is generic.

Elementary counting implies that for a rank $2$ distribution $D$, $\dim D^3(q)$ is at most $5$. One of the main results of the paper can be formulated as follows:
\begin{theorem}
\label{mainthm}
For any bracket generating rank $2$ distribution $D$ on an $n$-dimensional manifold $M$, $n>5$, with 
\begin{equation}
\label{D^3_eq}
\rank D^3=5
\end{equation}
at a generic point, the following statements hold:
\begin{enumerate}
\item One can assign to $D$ a canonical frame on the $(2n-1)$-dimensional bundle over a neighborhood of generic point of $M$ which implies that the group of symmetries of $D$ is at most $(2n-1)$-dimensional;
\item Any $(2,n)$-distribution with $(2n-1)$-dimensional group of symmetries is locally equivalent to the distribution associated with the Monge equation 
\begin{equation}
    \label{Monge}
    z'(x)=\bigl(y^{(n-3)}(x)\bigr)^2, 
\end{equation}
or equivalently, with a rank $2$ distribution on ${\mathbb R}^n$ with coordinates $(x,y_0,\ldots, y_{n-3},z_0)$ given by the intersection of the annihilators of the following $n-2$ one-forms:
\begin{equation}
 \label{Pfaff}
 \begin{aligned}
~&dy_i-y_{i+1} dx , \,\,0\leq i\leq n-4,\\
 ~&dz-y_{n-3}^2dx.
\end{aligned}
\end{equation}
\end{enumerate}
\end{theorem}
Note that for $n>5$ the infinitesimal symmetry algebra of the distribution associated with the Monge equation \eqref{Monge} is isomorphic to the natural semidirect sum of $\mathfrak{gl}_2(\R)$ with the $(2n-5)$-dimensional Heisenberg algebra
(for details, see \cite{doubrov2009local} or \cite{Day2025}).

\'{E}. Cartan proved an analogue of Theorem \ref{mainthm} for the case $n = 5$ \cite{FiveVariables}: the dimension of the bundle in item~(1) and the corresponding upper bound for the dimension of the symmetry algebra is 14 (equal to the dimension of the exceptional Lie algebra $G_2$). The maximally symmetric model in this case is given by equation~\eqref{Monge} (the Cartan–Hilbert equation) or, equivalently, by the Pfaffian system~\eqref{Pfaff} with $n = 5$, and its infinitesimal symmetry algebra is isomorphic to the split real form of $G_2$.

Theorem~\ref{mainthm} strengthens Theorems 1 and 3 of \cite{doubrov2009local} {\bf by removing an additional assumption that the distribution $D$ is of so-called maximal class} (see Definitions \ref{class_def} and \ref{max_class_def} in Section~\ref{sectionSymp} for the precise geometric definition.\footnote{A shorter optimal control description of this notion of maximal class via minimal corank of abnormal extremals is given by the conclusion of Theorem \ref{corank1_abnormal_thm} below, if one takes into account Remark \ref{strength_corank_rem}.}
Thus, Theorem~\ref{mainthm} is a direct consequence of those theorems from \cite{doubrov2009local} and of the following result, which forms the main technical core of the present paper:

\begin{theorem}
\label{max_class_conjecture}
    Any bracket generating rank $2$ distribution with $5$-dimensional cube is of maximal class at a generic point.
    The set of such points contains all points where the distribution is equiregular.
\end{theorem}

A precise description of the bundle and the canonical frame of item (1) of Theorem \ref{mainthm}, using the theory of Jacobi curves \cite{Zelenko_2006} and geometry of curves in projective spaces, can be found in \cite{doubrov2009local}. An alternative construction of this bundle and the canonical frame, using Tanaka-Morimoto theory, can be found in section 3 of \cite{Day2025}.

Theorem \ref{max_class_conjecture} was conjectured in \cite{DZ2006} (see also \cite{doubrov2009local}) nearly 20 years ago, but before the present paper, it had been confirmed only in the following very limited specific cases (see, e.g., \cite{Wendell}):
\begin{enumerate}
\item $5\leq n\leq 8$; 
\item for distributions with small grow vector $(2,3,5,6,7,\cdots, n)$ at every point; i.e., when $\dim D^j(q)=j+2$ for $4\leq j\leq n-2$;
\item in the case of $(2, 14)$-distributions with ``free" small growth vector $(2,3,5,8,14)$; 
\item for distributions associated with Monge equations 
\[
    z^{(m)} = F\big(x, y, y', \cdots, y^{(n-2-m)}, z,\cdots, z^{(m-1)}(n)\big), \quad \displaystyle{{\partial ^2 F\over \partial (y^{(n-2-m)})^2}\neq 0}.
\]
\end{enumerate}

Note that in \cite{AndersonKruglikov} it was demonstrated that the local model \eqref{Monge} is the most symmetric among Monge distributions. This was done by showing first that Monge distributions with fixed $1\leq m\leq \lfloor{\frac{n-2}{2}}\rfloor$ have fixed Tanaka symbol. In the sequel, those Tanaka symbols will be called \emph{Monge symbols}. Then by computing the universal Tanaka prolongation of such symbols, the authors find that the maximally symmetric model corresponds to $m=1$ and is locally equivalent to the model in \eqref{Monge}. However, for $n\geq 6$ generic germs of rank 2 distributions are not Monge \footnote{A heuristic explanation of this is that generic germs of $(2,n)$-distributions are described, up to local equivalence, by $\dim\mathrm {Gr}(2,n)-n=n-4$ functions of $n$ variables, whereas a Monge distribution is described by only one such function. Here, $\mathrm {Gr}(2,n)$ denotes the Grassmannian of planes in $\mathbb R^n$}: for $n\geq 7$ this follows from the fact that generic Tanaka symbols of $(2,n)$-distributions are not Monge symbols. For $n=6$ the statement follows from the fact that certain nontrivial invariants vanish for Monge distributions. \footnote{We believe that for $n\geq 6$ non-Monge distributions are generic even among distributions with Tanaka symbol isomorphic to a fixed Monge symbol.}

The previous approaches to proving Theorem \ref{max_class_conjecture} relied on attempts to compute the filtration \eqref{geod_flag} (below) directly in terms of the original distribution $D$. A key insight in \cite{Wendell} was that it suffices to prove Theorem~\ref{max_class_conjecture} for \emph{flat distributions} with prescribed Tanaka symbols \footnote{We also use this observation in the proof of the second sentence of Theorem \ref{max_class_conjecture}.} (for the definition of the Tanaka symbols and flat distributions, see \cite{Tanaka1970, Zelenko2009-qt} and also the end of section \ref{proof_sec}, after formula \eqref{Tanaka_comp}). However, this method depends on the classification of Tanaka symbols, which becomes infeasible due to the combinatorial explosion of possibilities as the dimension increases. To circumvent this obstacle, a two-stage strategy was proposed in \cite{Wendell}. The first stage involves computing the filtration \eqref{geod_flag} for free truncated nilpotent Lie algebras with two generators of a given step. The second stage aims to use the result of these computations for all nilpotent graded Lie algebras of the same step. However, the first stage quickly becomes computationally infeasible as the dimension grows, even for computer algebra systems. Moreover, even when the first stage was successfully completed (as in the 5-step case of item (3) above), it remained unclear how to use it to implement the second stage effectively.


In this paper, we adopt a completely different strategy,
which enables us to prove Theorem \ref{max_class_conjecture} in full generality, without relying on any computer algebra computations. Rather than attempting to compute the filtration \eqref{geod_flag} directly from the original distribution 
$D$ on $M$, we begin with an abstract filtration of the type \eqref{geod_flag} on the corresponding submanifold of the projectivized cotangent bundle $\mathbb PT^*M$. Assuming that the filtration enjoys the properties of one associated with a distribution of a given constant class greater than $1$ \footnote{A distribution with $5$ dimensional cube has class greater than $1$; see Lemma \ref{lemma: minimal class}.}, we show that it corresponds to a bracket-generating distribution only if the class is maximal. 
A key factor that convinced us of the promise of this method, and motivated us to pursue it, was the realization—based on general reasoning—that if the method were to fail, it would effectively yield a counterexample. Moreover, we recognized that the analysis should not strongly depend on the dimension, but from item (1) above, no counterexamples arise in low-dimensional cases. Finally, this approach is completely independent of the Tanaka symbols of the original distribution.

The next two subsections will expound the implications of Theorem \ref{mainthm} for non-Goursat rank 2 distributions and the implications of Theorem \ref{max_class_conjecture} for optimal control problems with constraint given by a rank 2 distribution with 5-dimensional cube, respectively.
\subsection{Canonical Frames for Non-Goursat Rank 2 Distributions}

We now explain why the assumption in equation \eqref{D^3_eq} is, in fact, not restrictive. In short, given a distribution which is not Goursat, one can apply iterative Cartan deprolongations (outlined below) at a generic point to reduce to a distribution satisfying \eqref{D^3_eq}. We thereby justify the title of the paper.

First, for $n = 3$ and $n = 4$, all generic germs of $(2, n)$-distributions are locally equivalent to each other, as established by the classical Darboux and Engel theorems. These distributions are modeled by the Cartan (or ``contact'') distributions on the jet spaces $J^1(\mathbb{R}, \mathbb{R})$ and $J^2(\mathbb{R}, \mathbb{R})$, respectively, where $J^k(\mathbb{R}, \mathbb{R})$ denotes the space of $k$-jets of functions from $\mathbb{R}$ to itself. For arbitrary $n \geq 3$, the Cartan distribution on $J^{n-2}(\mathbb{R}, \mathbb{R})$ provides a canonical model of a $(2, n)$-distribution, possessing an infinite-dimensional Lie algebra of symmetries given by the group of contact transformations. It is also worth noting that for $n \geq 4$, all such distributions have a $4$-dimensional cube. This shows that without assumption \eqref{D^3_eq} we may get models with infinite-dimensional symmetries.

On the other hand, even if \eqref{D^3_eq} does not hold at generic points, Theorem \ref{mainthm} is applicable (at a generic point) after a certain reduction procedure (see \cite[subsection 7.1]{doubrov2009local}) also called deprolongation in \cite[Remark 2.6]{doubrov2009local}. Indeed, because $D$ is bracket generating, its cube has dimension at least $4$ at generic points. Suppose that $D$ satisfies $\dim D^3(q)=4$ on an open set $M^\circ$ of $M$. Then the rank $3$ distribution $D^2$ on $M^\circ$ has a one-dimensional characteristic sub-distribution lying in $D$. At any $q_0\in M^\circ$, we can consider the quotient $\Pi:U\to \depr^1_{q_0}U$ of a neighborhood $U$ of $q_0$ by the corresponding one-dimensional foliation together with a new bracket generating rank $2$ distribution $\depr^1_{q_0}D$ on $\depr^1_{q_0}U$ obtained by the factorization of $D^2$. In fact, the germ of $D$ at $q_0$ can be uniquely reconstructed from $\depr_{q_0}^1D$, because it is equivalent to its Cartan prolongation (see \cite{Day2025} for details). Therefore $\depr^1_{q_0}D$ is called the \emph{(first) deprolongation of $D$ at the point $q_0$}. 

In the case that $\depr_{q_0}^1D$ has cube of rank 4 in a neighborhood of $\Pi(q_0)$, we can repeat this process at $\Pi(q_0)$ to obtain $\depr^1_{\Pi(q_0)}\big(\depr^1_{q_0}D\big)$, which we shall denote by $\depr^2_{q_0}D$ and call the \emph{second deprolongation of $D$ at $q_0$}. The $s$th deprolongation is defined inductively on the condition that $\depr^{s-1}_{q_0}D$ has cube of rank 4 in a neighborhood of image of $q_0$ under the corresponding composition of quotients. We denote this distribution by $\depr^s_{q_0}D$.

At a generic point $q_0$ of $M$, there exists $s$ so that after iterating this procedure $s$ times, one arrives at one of the following two cases:
\begin{enumerate}[label=(\roman*)]
\item The $(2,n-s)$ distribution $\depr^s_{q_0}D$ has $5$-dimensional cube in a neighborhood of the image of $q_0$ under the corresponding composition of quotients (so $s\leq n-5$).
\item The $(n-4)$th deprolongation $\depr^{n-4}_{q_0}D$ is well defined (and in this case is locally equivalent to the Engel distribution).
\end{enumerate}

We call the $s$ appearing in case (i) the \emph{deprolongation degree} of $D$ at $q_0$. We also set the deprolongation degree for case (ii) equal to $n-4$. Note that the deprolongation degree is defined at a generic point of $M$, but in general, this degree is only locally constant. It is not greater than $n-5$ in case (i).

In case (ii), the original distribution $D$ is locally equivalent to a Goursat distribution \cite{MZ2010} at $q_0$. Goursat distributions are defined using the \emph{strong derived flag} $\{D^{[j]}\}_{j=1}^\infty$, where instead of \eqref {week_derived} we have
\begin{equation}
\label{strong_derived}
    D^{[1]}(q) = D(q)\quad \text{and}\quad D^{[i]}(q) = D^{[i-1]}(q)+[D^{[i-1]},D^{[i-1]}](q)\ \text{for } i>1.
\end{equation}
The distribution $D$ is called \emph{Goursat} if $\dim D^{[i]}(q)=i+1$ for every $i\geq 1$ and every $q\in M$. It is very well known \cite{MZ2010} that any $(2,n)$ Goursat distribution at a generic point is locally equivalent to the Cartan distribution on $J^{n-2}(\mathbb R, \mathbb R)$ and therefore the germs at such points have an infinite-dimensional infinitesimal symmetry algebra.

In case (ii), the germ of $D$ at $q_0$ is Goursat. In case (i) our original distribution is not Goursat near $q_0$, and we can apply Theorem \ref{mainthm} or Cartan's result in \cite{FiveVariables} to $\depr_{q_0}^sD$ with $n$ replaced by $n-s$, namely:

\begin{theorem}
\label{non_Goursat_theorem}
Let $D$ be a bracket-generating rank $2$ distribution on an $n$-dimensional manifold $M$, which is nowhere Goursat. For a point $q_0$ for which the deprolongation degree $s$ is defined
\footnote{Note that the deprolongation degree is defined at a generic point. Since $D$ is nowhere Goursat, we have that $s\leq n-5$.}, the following two statements are true:
\begin{enumerate}
    \item If $s<n-5$, then one can canonically assign to $\depr_{q_0}^s D$ a frame on a $(2n - 2s - 1)$-dimensional bundle over a neighborhood of a generic point in the ambient manifold of $\depr_{q_0}^s D$. In particular, this implies that the symmetry group of the germ of $D$ at $q_0$ has dimension at most $2n - 2s - 1$.
    \item If $s=n-5$, then one can canonically assign to $\depr_{q_0}^s D$ a frame on a $14$-dimensional bundle over a neighborhood of a generic point in the ambient manifold of $\depr_{q_0}^s D$. In particular, this implies that the symmetry group of the germ of $D$ at $q_0$ has dimension at most $14$.
\end{enumerate}
In either case, the maximally symmetric germ among $(2,n)$ with deprolongation degree $s$ is locally equivalent to the generic germ of the $s$th Cartan prolongation of the rank $2$ distribution associated with the Monge equation
    \begin{equation}
        \label{Monge_k}
        z'(x) = \left( y^{(n - 3 - s)}(x) \right)^2,
    \end{equation}
    or, equivalently, with the rank $2$ distribution on $\mathbb{R}^{n-s}$ given by the Pfaffian system \eqref{Pfaff}, with $n$ replaced by $n-s$.
\end{theorem}

The last theorem together with the classical theory of Goursat distributions covers in principle (i.e., modulo analysis of the invariants coming from the canonical frame in non-Goursat case) the local geometry of all bracket generating rank 2 distributions in a neighborhood of generic points: 
In the Goursat case, the germs at generic points are all locally equivalent. In the non-Goursat case, after a suitable number of deprolongations, one can construct a canonical structure of absolute parallelism.

We emphasize that our analysis concerns \emph{arbitrary} bracket generating rank $2$ distributions, but only in neighborhoods of \emph {generic} points, where the notion of genericity is explicitly defined by the assumption that the class of the distribution (or its appropriate deprolongation) at a point is maximal (see Definition \ref{max_class_def} below).

We do not address the equivalence problem for germs of distributions at points that do not satisfy this genericity condition (which we refer to as singular points). Even in the case of Goursat distributions, the moduli space of all germs is quite wild and consists of all germs appearing in the so-called Monster Tower (see \cite{MZ2010})—even though the generic germ is unique up to local equivalence.

In the non-Goursat case, the situation is much more intricate: even generic germs admit functional invariants coming from the canonical frames. Nonetheless, if these function invariants are nontrivial, the canonical frame constructed at generic points may still provide valuable information about the equivalence of germs at singular points, for instance, by studying how the invariants of the distribution behave as one approaches these singularities.

\subsection{The Existence of Corank 1 Abnormal Extremals through Generic Points}
Finally, note that Theorem \ref{max_class_conjecture} is of independent interest from the perspective of optimal control theory. 
On the space of Lipschitz curves that are almost everywhere tangent to a given distribution $D$, called \emph{horizontal curves of $D$}, consider any variational problem that assigns a cost to each such curve—for example, the problem of minimizing length with respect to a sub-Riemannian metric. The Pontryagin Maximum Principle \cite{Pontryagin, Agrachev_Sachkov_2004} characterizes, through the Hamiltonian formalism, a class of curves, known as Pontryagin extremal trajectories, among which the minimizers of such problems (with fixed endpoints) must lie. Extremal trajectories for which the Lagrange multiplier which multiplies the cost vanishes are called \emph{abnormal extremal trajectories}; they depend only on the distribution $D$ and not on the specific cost functional, and they are also called \emph{singular curves of the distribution $D$} \cite{Montgomery_book}. 

While abnormal extremal trajectories can be described purely geometrically using the canonical symplectic form on the cotangent bundle of $M$ (\cite{Liu-Sussmann, Montgomery_book,Zelenko99}, see also subsection \ref{subsec: Char Line Distr} below), for brevity we use here an equivalent description as critical points of the endpoint map: Given a point $q_0$ and a time $T$, denote by $\Omega_{q_0}(T)$ the set of all horizontal curves of $D$ starting at $q_0$ defined on $[0,T]$, and by $F_{q_0, T} : \Omega_{q_0}(T) \to M$ the \emph{endpoint map} that takes each $\gamma \in \Omega_{q_0}(T)$ to the endpoint $\gamma(T)$. Note that if the set $\Omega_{q_0}(T)$ has the structure of a $(L^\infty([0,T]))^l$-manifold, where $l$ is the rank of $D$. 
\begin{definition}
\label{abn_def}
A horizontal curve $\gamma : [0, T] \to M$ is an \emph{abnormal extremal trajectory} of the distribution $D$ if it is a critical point of the mapping $F_{q_0,T}$, that is, if $\operatorname{Im} \left( d(F_{q_0,T})_\gamma \right) \neq T_{\gamma(T)} M$, where $d (F_{q_0,T})_\gamma$ denotes the differential of the endpoint map $F_{q_0,T}$ at $\gamma$. The \emph{corank} of the abnormal extremal trajectory $\gamma$ is defined as the codimension of $\operatorname{Im} \left( D_\gamma F_{q_0,T} \right)$ in $T_{\gamma(T)}$ and is denoted by $\mathrm{corank}(\gamma)$. 
\end{definition}
Obviously, the corank of abnormal extremal trajectory is at least $1$. Theorem \ref{max_class_conjecture} implies the following theorem:

\begin{theorem}
\label{corank1_abnormal_thm}
    Given a bracket generating rank $2$ distribution $D$ with $5$-dimensional cube on a manifold $M$, for a generic point $q_0\in M$ there exists an abnormal extremal trajectory of corank $1$ starting at $q_0$. The set of such points contains all points where the distribution is equiregular. 
\end{theorem}
\begin{remark}
\label{strength_corank_rem}
In fact, Theorem \ref{max_class_conjecture} implies 
a stronger statement: first, such an abnormal extremal is regular in a sense of \cite{Liu-Sussmann} (equivalently, satisfies the generalized Legendre-Glebsch condition in terminology of \cite{Agrachev_Sachkov_2004, Zelenko99}); and second, there are plenty of such abnormal extremal trajectories; see Theorem \ref{corank1_abnormal_thm_reg} below.
\end{remark}
We conjecture that Theorem \ref{corank1_abnormal_thm} (and its stronger version, Theorem \ref{corank1_abnormal_thm_reg} below) are valid at any point--that is, the phrase ``at a generic point" can be omitted—although this remains beyond our current reach. See Remark \ref{q_0_nonmax_class_rem} below for discussion of this issue.

The paper is organized as follows: In Section \ref{sectionSymp} we recall the definition of the class of a distribution from \cite{doubrov2009local} which is used in Theorem \ref{max_class_conjecture}. In Section \ref{proof_sec} we prove Theorem \ref{max_class_conjecture} and therefore Theorems \ref{mainthm} and \ref{non_Goursat_theorem}. Finally, in section \ref{corank_sec} we prove a stronger version of Theorem \ref{corank1_abnormal_thm}, which is labeled Theorem \ref{corank1_abnormal_thm_reg}.

{\bf Acknowledgment} We would like to thank Boris Doubrov for several valuable comments, which clarified and simplified some arguments.

\section{Rank 2 distributions of Maximal class}\label{sectionSymp}

The \emph{class} of a distribution was defined in \cite{DZ2006, doubrov2009local} in the development of the so-called symplectification procedure. 
For completeness, we define the notion of class here, reproducing the constructions of our previous work \cite{Day2025}, which deviates only modestly from the original construction. This is especially important, because in contrast to \cite{Day2025,DZ2006,doubrov2009local} , which were focused mainly on the case of maximal class, here we need to consider an arbitrary class.

The symplectification procedure utilizes the natural contact structure on the projectivized cotangent bundle $\mathbb{P}T^*M$ to construct a so-called ``even contact structure'' on a submanifold $\mathcal{M}\subseteq \mathbb{P}T^*M$ of codimension 3. The kernel of this even contact structure is a canonical line distribution $\mathcal{C}$. Lifting $D$ to $\mathbb{P}T^*M$ and osculating with $\mathcal{C}$ yields a flag at each point of $\mathbb{P}T^*M$. The class of the distribution will then be defined by considering the growth of the dimensions associated with this flag. 

 \subsection{The Characteristic Line Distribution and Regular Abnormal Extremals}
 \label{subsec: Char Line Distr}
 Let $D$ be a bracket generating rank 2 distribution on a smooth manifold $M$ of dimension $n\geq5$. Define the annihilator of $D^{\ell}$
 \[
     \big(D^{\ell}\big)^\perp = \big\{ (p,q)\in T^*M : p\cdot v = 0\ \forall\ v\in D^{\ell}(q)\big\}.
 \]
 
 Consider the fiberwise projectivization $\mathbb{P}T^*M$ of the cotangent bundle.
 Since each $(D^{\ell})^\perp$ is a linear subbundle of $T^*M$, we may define a codimension 3 submanifold
 \[
     \mathcal{M} = \mathbb{P}\Big((D^{2})^\perp\setminus (D^{3})^\perp\Big)\subseteq \mathbb{P}T^*M.
 \]
 Let $\mathfrak{s}$ be the tautological (Liouville) one-form on $T^*M$; explicitly, for coordinates $(q^i)$ on $M$ with conjugate variables $p_i$, $\mathfrak{s}=\sum p_i\text{d}q^i$. Recall that $d\mathfrak{s}$ is the canonical symplectic form on $T^*M$. The form $\mathfrak{s}$ passes to a conformal class $\oline{\mathfrak{s}}$ of 1-forms on $\mathbb{P}T^*M$ which defines a contact structure.
 
 Since $\text{rank}(D^{2})=3$, the submanifold $\mathcal{M}$ has codimension 3 in the contact manifold $\mathbb{P}T^*M$. Restricting the contact forms $\oline{\mathfrak{s}}$ to $\mathcal{M}$ gives a hyperplane distribution \begin{equation} \label{evencontact_H} H=\ker\left(\oline{\mathfrak{s}}|_{\mathcal{M}}\right) \end{equation}
 with a conformal class of skew-symmetric forms $\oline{\sigma}=d\oline{\mathfrak{s}}|_{H}$ well-defined on this hyperplane distribution. Since $H$ is a hyperplane distribution, it has rank $2n-5$, so the kernel of the form $\oline\sigma$ must have odd rank. In \cite{doubrov2009local}, the authors show that $\ker(\oline{\sigma})$ has the minimal rank of 1, so that $\mathcal{M}$ is equipped with a so-called \emph{even contact structure}. We shall write $\mathcal{C}$ for the line distribution $\ker(\oline{\sigma})$, called the \emph{characteristic line distribution} of $D$. Following (\cite{Liu-Sussmann,Agrachev-Sarychev, Zelenko99}) 
 \begin{definition}
 \label{reg_abn_def}
 The integral curves of $\mathcal{C}$ are called \emph{regular abnormal extremals} of the distribution $D$, and their projections onto $M$ are \emph{regular abnormal extremal trajectories}. 
\end{definition} 
The reason why regular abnormal extremal trajectories are indeed abnormal extremal trajectories in the sense of Definition~\ref{abn_def} is explained at the beginning of the proof of Theorem \ref{corank1_abnormal_thm_reg} in Section \ref{corank_sec}. In particular, see relation \eqref{diff_end_point_jac}.

 \subsection{Maximality of Class}
 \label{sectionSympD} Let $\pi:\mathcal M\to M$ be the canonical projection. The lift of $D$ to $\mathcal{M}$ is denoted by:
 \begin{equation}
 \label{J0}
     \mathcal{J}(\lambda) = \big\{v\in T_\lambda \mathcal{M} : \pi_*(v)\in D\bigl(\pi(\lambda)\bigr)\big\}
 \end{equation}
 which is a distribution of rank $n-2$. Osculating with the characteristic line distribution $\mathcal{C}$, we obtain from $\mathcal{J}$ a flag at each point of $\mathcal{M}$. Write $\mathcal{J}^{(0)} = \mathcal{J}$ and define recursively
 \begin{equation}
 \label{geod_flag}
     \mathcal{J}^{(i)}(\lambda) = \mathcal{J}^{(i-1)}(\lambda) + [\mathcal{C}, \mathcal{J}^{(i-1)}](\lambda)\quad \text{for}\ i\geq 1
 \end{equation}
 In Proposition 1 of the paper \cite{doubrov2009local}, the authors show that for each $0\leq i$ and each $\lambda\in \mathcal{M}$, we have 
 \begin{equation}
 \label{jump}
 \text{dim}\big(\mathcal{J}^{(i+1)}(\lambda)\big) - \text{dim}\big(\mathcal{J}^{(i)}(\lambda)\big) \leq 1.
 \end{equation}
 so that 
 \begin{equation}
 \label{max_dim_i}
 \dim
 \mathcal{J}^{(i)}(\lambda) \leq n-2+i.
 \end{equation}
 Let $H$ be as in \eqref{evencontact_H}. Since $\mathcal{C}$ is a Cauchy characteristic of $H$, and since $\mathcal{J}$ is contained within this distribution, so too is each $\mathcal{J}^{(i)}$. Hence 
 \begin{equation}
 \label{max_dim}
 \dim\mathcal{J}^{(i)}(\lambda)\leq \text{rank}(H)= 2n-5. 
 \end{equation}
 Define the integer-valued functions on $\mathcal{M}$ and $M$, respectively:
 \begin{gather}
     \label{abn_ext_class}
     \nu(\lambda) = \min\{i\in \N: \mathcal{J}^{(i+1)}(\lambda) = \mathcal{J}^{(i)}(\lambda)\}
     \\
     \label{class_eq}
     m(q) = \max\{\nu(\lambda): \lambda\in \pi^{-1}(q)\}
 \end{gather}
 One can show that the set $\{\lambda\in \pi^{-1}(q): \nu(\lambda)=m(q)\}$ is nonempty and Zariski open in the fiber $\pi^{-1}(q)$ and that $\nu(\cdot)$ and $m(\cdot)$ are lower semicontinuous. 
 
 \begin{definition}
 \label{class_def}
 The value $\nu(\lambda)$ defined by \eqref{abn_ext_class} is called the \emph{class at $\lambda$} of the regular abnormal extremal passing through $\lambda$. 
 The value $m(q)$ defined by \eqref{class_eq} is called the \emph{class of $D$ at $q$}. 
 \end{definition}
 The relation between the class at $\lambda$ of the regular abnormal extremal passing through $\lambda$ and the corank of the corresponding abnormal extremal trajectory is given in section \ref{corank_sec}; see \eqref{class_corank_ineq_3}.
 Note that \eqref{max_dim_i} and \eqref{max_dim} imply that $\nu(\lambda)\leq n-3$, and therefore $m(q)\leq n-3$. The equality $\nu(\lambda)=n-3$ holds if and only if $\mathcal{J}^{(n-3)}(\lambda)=H(\lambda)$. 
 \begin{definition}
 \label{max_class_def}
We say that $D$ is \emph{of maximal class} at $q\in M$, if its class $m(q)$ is equal to $n-3$. On the other hand, we say that $D$ is \emph{of minimal class} at $q\in M$ if $m(q)=1$.
\end{definition}
  The following lemma is proven as Remark 3.4 in \cite{Zelenko_2006} and will be used in the proof of Theorem \ref{max_class_conjecture}.
 \begin{lemma}
    \label{lemma: minimal class}
    Let $D$ be a bracket generating rank 2 distribution on an $n$-dimensional manifold $M$, $n>5$. For each $q\in M$, $D$ is of minimal class at $q$ if and only if $D^{3}(q)$ has dimension $4$.
 \end{lemma}
 Proposition 3.4 of \cite{Zelenko_2006} demonstrates that germs of $(2,n)$ distributions of maximal class are generic. Theorem \eqref{max_class_conjecture} is a much stronger statement that we want to prove.
 
 Note that the class of the distribution is constant in a neighborhood of a generic point of $M$. To prove the first sentence of Theorem \ref{max_class_conjecture}, it suffices to restrict our considerations to such neighborhoods, as we do in Theorem \ref{class_constant} below. Therefore, instead of introducing special notation for these neighborhoods, we will, from now on, assume that the distribution $D$ has constant class $m$ on $M$. Then the set
 \[
     \mathcal{R}_{m}=\{\lambda\in \mathcal{M}: \nu(\lambda)=m\}.
 \]
 is open and dense in the space $\mathcal{M}$.

\subsection{Involutivity Conditions}
Now let $D$ be a rank 2 distribution of constant class $m$ on an $n$-dimensional manifold $M$, $n\geq5$. Then by \eqref{jump} for any $q\in M$ and any $\lambda\in \mathcal{R}_{m}$, the flag
 \[
     \mathcal{J}(\lambda)\subseteq\mathcal{J}^{(1)}(\lambda)\subseteq \cdots\subseteq \mathcal{J}^{(m)}(\lambda) \subseteq H(\lambda)\subseteq T_\lambda\mathcal{R}_{m}
 \]
 has the property that $\text{rank}(\mathcal{J}^{(i+1)})= \text{rank}(\mathcal{J}^{(i)})+1$ for each $0\leq i \leq m-1$. Further, from the assumption that the class is equal to $m$ it follows that
 \begin{equation}
    \label{X_stab}[\mathcal{C},\mathcal{J}^{(m)}]\subseteq \mathcal{J}^{(m)}.
 \end{equation}
 We can use the conformal class of 2-forms $\oline{\sigma}$ defined in section \ref{subsec: Char Line Distr} to continue this flag: for each $i\geq 1$ and each $\lambda\in \mathcal{R}_{m}$, define
 \[
     \mathcal{J}_{(i)}(\lambda) = \big\{v\in T_\lambda \mathcal{R}_{m} : \oline{\sigma}(v,w) = 0\ \forall\ w\in \mathcal{J}^{(i)}(\lambda)\big\},
 \]
 the skew complement of $\mathcal{J}^{(i)}(\lambda)$ with respect to $\oline\sigma$. From the definition of $\mathcal{M}$, it follows quickly that $\oline{\sigma}\left(\mathcal{J},\mathcal{J}\right) = 0$; we obtain a flag:
 \begin{equation}
     \label{calJFlag}
     \mathcal{C}(\lambda)\subseteq \mathcal{J}_{(m)}(\lambda)\subseteq \cdots \subseteq \mathcal{J}_{(1)}(\lambda)\subseteq \mathcal{J}(\lambda)\subseteq \mathcal{J}^{(1)}(\lambda)\subseteq \cdots \subseteq \mathcal{J}^{(m)}(\lambda)\subseteq H(\lambda)\subseteq T_\lambda\mathcal{R}_{m}.
 \end{equation}
 By \eqref{jump}, we have for each $0 < i \leq m$ that
 \begin{equation}
    \label{flag dims}
     \dim\big(\mathcal{J}_{(i)}(\lambda)\big) = n-2-i\quad\text{and}\quad \dim\big(\mathcal{J}^{(i)}(\lambda)\big) = n-2+i.
 \end{equation}
 
 At each $\lambda\in \mathcal{R}_{m}$, one can show that
 \[
    \mathcal{J}^{(1)}(\lambda) = \big\{v\in T_\lambda \mathcal{M} : \pi_*(v)\in D^{2}(\lambda)\big\}
 \]
 which implies (with some computation) that $\mathcal{J}_{(1)}(\lambda) = \ker(T_\lambda\pi)\oplus \mathcal{C}(\lambda)$. Define $V_1(\lambda) = \ker(T_\lambda\pi)$, the vertical subspace over $\lambda$. For $i=0$ and for $2\leq i\leq m,$ define
 \begin{equation}
 \label{Splitting}
     V_i(\lambda) = \mathcal{J}_{(i)}(\lambda)\cap V_1(\lambda),
 \end{equation}
 the vertical component of the $(n-2-i)$-dimensional piece of the flag at $\lambda$. From \eqref{calJFlag} it is clear that 
 \begin{equation}
\label{vert_inclusion}
V_{i}(\lambda)\subset V_{i-1}(\lambda), \quad i\geq 1.
 \end{equation}
Further, observe that $V_0(\lambda)=V_1(\lambda)$. From \eqref{calJFlag}, one can also observe that for each $1\leq i\leq m$
 \begin{equation}
    \label{J_i = C + V_i}
    \mathcal{J}_{(i)}(\lambda) = V_i(\lambda)\oplus \mathcal{C}(\lambda),
 \end{equation}
 The $V_i$ also satisfy involutivity conditions; in the following proposition, which is Lemma 2 in \cite{doubrov2009local}, recall that $V_0=V_1$.
\begin{prop}
\label{prop: involutivity}
    For each $q\in M$ and any $0\leq i \leq m$, we have involutivity conditions
    \begin{gather}
    \label{invol1}
        [V_i,V_i] \subseteq V_i
        \\
    \label{invol2}
        [V_i,\mathcal{J}^{(i)}] \subseteq \mathcal{J}^{(i)}
    \end{gather}
for the flag of distributions on $\mathcal{R}_{m}$
\end{prop}
Although the statement is only proven for distributions of maximal class in \cite{doubrov2009local}, the proof does not rely on maximality of class. In order to prove the main theorem, we shall need one more fact about the flag \eqref{calJFlag}. 

The following lemma is a small extension of Remark 2 of \cite{doubrov2009local}. In that remark, only one inclusion of the equality \eqref{lemma Alt J_i} is demonstrated. Later in that paper, the equality is shown assuming maximality of class. We provide a detailed proof here because the lemma is crucial in the proof of Theorem \ref{max_class_conjecture}.
\begin{lemma}
    \label{lemma Alt J_i}
    For each $1\leq i \leq m-1$ and each $\lambda\in \mathcal{R}_{m}$, we have that
    \begin{equation}
        \label{Alt J_i}
         \mathcal{J}_{(i)}(\lambda)+ [\mathcal{C},\mathcal{J}_{(i)}](\lambda) = \mathcal{J}_{(i-1)}(\lambda)
    \end{equation}
    Further,
    \begin{equation}
        \label{Alt J_m}
          [\mathcal{C},\mathcal{J}_{(m)}] (\lambda) \subseteq \mathcal{J}_{(m)}(\lambda).
    \end{equation}
\end{lemma}
\begin{proof}
    Fix $\lambda\in \mathcal{R}_{m}$ and choose a nonvanishing local section $X$ of $\mathcal{C}$ near $\lambda$. Also choose a skew-symmetric 2-form $\sigma$ on $\mathcal{R}_{m}$ from the conformal class $\oline{\sigma}$ defined in Section \ref{subsec: Char Line Distr}. Since the Lie derivative $L_X\sigma=f\sigma$ for some $f\in C^\infty(\mathcal{R}_{m})$, we have for any sections $Y$ and $Z$ of the contact distribution $H$ satisfying $\sigma(Y,Z)\equiv 0$ that
    \[
        \sigma\big([X,Y],Z\big) = -\sigma\big(Y,[X,Z]\big)
    \]
    
    Now fix $1\leq i \leq m-1$; let us begin with the rightward inclusion of \eqref{Alt J_i}. Fix a section $Y$ of $\mathcal{J}_{(i)}$ and a section $Z$ of $\mathcal{J}^{(i-1)}$. Since $\mathcal{J}_{(i)}\subseteq \mathcal{J}_{(i-1)}$, we have $\sigma(Y,Z)\equiv 0$, so that
    \[
        \sigma\big([X,Y],Z\big) = -\sigma\big(Y,[X,Z]\big) = 0
    \]
    where the last equality holds because $[X,Z]$ is a section of $\mathcal{J}^{(i)}$. This demonstrates the rightward inclusion for \eqref{Alt J_i}.

    For the leftward inclusion of (\ref{Alt J_i}), we show for arbitrary $\lambda\in \mathcal{R}_{m}$ that
    \begin{equation}
    \label{[C,J_i] not in J_i}
        \mathcal{J}_{(i)}(\lambda)+ [\mathcal{C},\mathcal{J}_{(i)}](\lambda) \not\subseteq \mathcal{J}_{(i)}(\lambda).
    \end{equation}
    so that the conclusion follows by the rightward inclusion of \eqref{Alt J_i} and \eqref{flag dims}. By \eqref{flag dims}, we can choose a section $Z$ of $\mathcal{J}^{(i)}$ so that $[X,Z](\lambda)\in \mathcal{J}^{(i+1)}(\lambda)\setminus\mathcal{J}^{(i)}(\lambda)$. Since $[X,Z](\lambda)\notin \mathcal{J}^{(i)}$, we can find a section $Y$ of $\mathcal{J}_{(i)}$ so that $\sigma\big(Y(\lambda),[X,Z](\lambda)\big) \neq 0$. Since $\sigma(Y,Z)\equiv 0$, we then have
    \[
        \sigma\big([X,Y](\lambda),Z(\lambda)) = -\sigma\big(Y(\lambda),[X,Z](\lambda)\big) \neq 0
    \]
    Thus $[X,Y](\lambda)$ is not in $\big(\mathcal{J}^{(i)}(\lambda)\big)^\angle=\mathcal{J}_{(i)}(\lambda)$, and we have proven \eqref{Alt J_i}.

    Finally, let us prove \eqref{Alt J_m}. For any section $Y$ of $\mathcal{J}_{(m)}$ and any section $Z$ of $\mathcal{J}^{(m)}$, we have that $\sigma(Y,Z)\equiv 0$, so that
    \[
        \sigma\big([X,Y],Z\big) = -\sigma\big(Y,[X,Z]\big) = 0
    \]
    where the last equality follows because $[\mathcal{C},\mathcal{J}^{(m)}]\subseteq \mathcal{J}^{(m)}$.
\end{proof}

\section{Proof of the Theorem \ref{max_class_conjecture}}
\label{proof_sec} 

As was noted at the end of section \ref{sectionSympD}, the class of a rank 2 distribution is constant in a neighborhood of a generic point of the base manifold. Therefore to prove the first sentence of Theorem \ref{max_class_conjecture}, it suffices to prove the following 
\begin{theorem}
\label{class_constant} If a bracket generating rank $2$ distribution with $5$-dimensional cube has constant class, then it is of maximal class. 
\end{theorem}
\begin{proof} Assume that $D$ has constant class $m$. Since $\rank D^3 = 5$, we have by Lemma \ref{lemma: minimal class} that $m>1$. Consider again the flag \eqref{calJFlag} on $\mathcal{R}_{m}$. Each piece of this flag has constant rank, as does $V_i$ for each $0\leq i \leq m$. Now fix a $\lambda_0$ in $\mathcal{R}_{m}$. Let 
\begin{equation}
\label{E_def}
E:=\mathcal J_{(m-1)}.
\end{equation}

Choose a nonvanishing section $X$ of $\mathcal{C}$ near $\lambda_0$. By \eqref{J_i = C + V_i}, we have that $E=\mathcal{C}\oplus V_{m-1}$. Choose a section $\ve_1$ of $V_{m-1}$ so that $[X,\ve_1](\lambda_0) \notin\mathcal{J}_{(m-1)}(\lambda_0)$. Write
\[
    \mathrm{pr}_V: \mathcal{J}_{(1)}=\mathcal{C}\oplus V_1\to V_1
\]
for the projection onto the vertical subspace $V_1$ parallel to $\mathcal{C}$. For each $2\leq i \leq 2m$, define 
\begin{equation}
    \ve_i= \begin{cases}
        \mathrm{pr}_V([X,\ve_{i-1}]) & \text{if }2\leq i \leq m-1
        \\
        [X,\ve_{i-1}] &\text{if }m\leq i\leq 2m
    \end{cases}
\end{equation}
Then
\begin{align}
    &\mathcal{J}_{(i)} = \mathcal{J}_{(m)}\oplus\langle \ve_1,\ldots, \ve_{m-i}\rangle\text{ for all }0\leq i \leq m
    \\
    &\mathcal{J}^{(i)} = \mathcal{J}_{(m)}\oplus \langle \ve_1,\ldots, \ve_{m+i}\rangle\text{ for all }0\leq i \leq m
\end{align}
Notice that by Proposition \ref{prop: involutivity}, we have 
\begin{equation}
    \label {w_i involutivity}
    V_{i} = V_{m}\oplus \langle\ve_1,\ldots, \ve_{m-i}\rangle\text{ is involutive for each }1\leq i \leq m
\end{equation}
Since $D$ is bracket generating, so too is the distribution $\mathcal{J}_{(0)}$ ($=\mathcal J$) on $T\mathcal{R}_{m}$, as $\mathcal J$ is the lift of $D$. By \eqref{Alt J_i}, this implies that the distribution $E$, defined by \eqref{E_def}, is also bracket generating. We claim that
\begin{equation}
\label{E osc flag}
    E^{i} = \begin{cases}
        \mathcal{J}_{(m-i)} & \text{for }1\leq i \leq m
        \\
        \mathcal{J}^{(i-m)}& \text{for }m+1\leq i \leq 2m
    \end{cases}
\end{equation}
The claim holds for $i=1$ by construction of $E$. To prove the claim for $2\leq i\leq m$, use induction, relations \eqref{J_i = C + V_i}, \eqref{vert_inclusion}, \eqref{Alt J_i}, and the first involutivity condition \eqref{invol1} of Proposition \ref{prop: involutivity}, to get:
\begin{gather*}
    E^i = E^{i-1}+[E,E^{i-1}] 
    = \mathcal{J}_{(m-i+1)}+[\mathcal{J}_{(m-1)},\mathcal{J}_{(m-i+1)}] 
    \\
\stackrel{\eqref{J_i = C + V_i}}{=} \mathcal{J}_{(m-i+1)}+[\mathcal{C}\oplus V_{m-1}, \mathcal{C}\oplus V_{m-i+1}] \stackrel{\eqref{vert_inclusion} \&\eqref{invol1}}{=} \mathcal{J}_{(m-i+1)}+[\mathcal C,\mathcal{J}_{(m-i+1)}]\stackrel{\eqref{Alt J_i}}{=}
\mathcal{J}_{(m-i)},
\end{gather*}
where in the first line we used the induction hypothesis. Similarly, for each $m+1\leq i \leq 2m$, use induction, relations \eqref{J_i = C + V_i}, \eqref{vert_inclusion}, \eqref{geod_flag}, and the second involutivity condition \eqref{invol2} of Proposition \ref{prop: involutivity} to get
\begin{gather*}
    E^i = E^{i-1}+[E,E^{i-1}] = \mathcal{J}^{(-m+i-1)} + [\mathcal{J}_{(m-1)},\mathcal{J}^{(-m+i-1)}] 
    \\
    \stackrel{\eqref{J_i = C + V_i}}{=}\mathcal{J}^{(-m+i-1)} + [\mathcal{C}\oplus V_{m-1},\mathcal{J}^{(-m+i-1)}] \stackrel{\eqref{vert_inclusion} \&\eqref{invol2}}{=}\mathcal{J}^{(-m+i-1)} + [\mathcal{C},\mathcal{J}^{(-m+i-1)}] \stackrel{\eqref{geod_flag}}{=}\mathcal{J}^{(-m+i)},
\end{gather*}
where again in the first line we used the induction hypothesis.
This demonstrates \eqref{E osc flag}.

Now define 
\begin{equation}
\label{eta_def}
\eta = [\ve_1,\ve_{2m}].
\end{equation}
Because $\ve_1(\lambda)\notin\mathcal{J}_{(m)}(\lambda)= \big(\mathcal{J}^{(m)}(\lambda)\big)^\angle$ for each $\lambda\in\mathcal{M}$, we have that $\eta(\lambda)\notin H(\lambda)$, where $H$ is the even contact distribution defined in \eqref{evencontact_H}. Therefore, $\eta(\lambda)\notin\mathcal{J}^{(m)}(\lambda)\subseteq H(\lambda)$. Since $D$ has constant class $m$, we have that $[\mathcal{C},\mathcal{J}^{(m)}]\subseteq \mathcal{J}^{(m)}$. Along with the second involutivity condition \eqref{invol2} of Proposition \ref{prop: involutivity}, this implies
\[
    E^{2m+1} = [E,\mathcal{J}^{(m)}] = [V_{m}\oplus \langle X,\ve_1\rangle, \mathcal{J}^{(m)}] = \mathcal{J}^{(m)}\oplus \langle \eta\rangle.
\]

We ultimately aim to show that 
\[
    [E,E^{2m+1}]\subseteq E^{2m+1},
\]
so that $E^{2m+1}$ is involutive. To demonstrate this, we prove a sequence of lemmas.

\begin{lemma}
    \label{lemma (i)}
    For each $1\leq i\leq m$,
    \[
        [\ve_i,\ve_{2m+1-i}] \equiv (-1)^{i+1}\eta \mod E^{2m}
    \]
\end{lemma}

\begin{proof}
The claim holds by definition \eqref{eta_def} of $\eta$ for $i=1$ . Now for the induction step, apply the Jacobi identity to obtain
\begin{align*}
    [\ve_{i},\ve_{2m+1-i}] &\equiv \big[[X,\ve_{i-1} \text{ mod } \langle X \rangle],\ve_{2m+1-i}\big] 
    \\
    &\equiv \big[X,[\ve_{i-1},\ve_{2m+1-i}]\big]+ \big[[X,\ve_{2m+1-i}],\ve_{i-1}\big]
    \\ 
    &\equiv \big[X,[\ve_{i-1},\ve_{2m+1-i}]\big] + [\ve_{2m+2-i},\ve_{i-1}] 
    \\
    &\equiv \big[X,[\ve_{i-1},\ve_{2m+1-i}]\big] + (-1)^{i+1}\eta
    \\
    &\equiv (-1)^{i+1}\eta
    \mod E^{2m}
\end{align*}
where the last equivalence holds because, first, \eqref{invol2} implies $[\ve_{i-1},\ve_{2m+1-i}]\subset \mathcal{J}^{(m)}$ and, second, by \eqref{Alt J_m}, we have
\[
    \big[X,[\ve_{i-1},\ve_{2m+1-i}]\big]\subseteq [\mathcal{C},\mathcal{J}^{(m)}]\subseteq \mathcal{J}^{(m)} = E^{2m}.
\]
\end{proof}

\begin{lemma}
    \label{lemma (ii)}
    \[
        [\ve_2,\ve_{2m}] \equiv (1-m)[X,\eta]\mod E^{2m+1}
    \]
\end{lemma}
\begin{proof}
    Observe that for each $2\leq i \leq m+1$, the Jacobi identity and Lemma \ref{lemma (i)} give
    
    \begin{align*}
        [\ve_i,\ve_{2m+2-i}] &\equiv \big[\ve_{i},[X,\ve_{2m+1-i}]\big]
        \\
        &\equiv \big[[\ve_{i},X],\ve_{2m+1-i}\big] + \big[X,[\ve_i,\ve_{2m+1-i}]\big]
        \\
        &\equiv -[\ve_{i+1} \text{ mod } \langle X\rangle ,\ve_{2m+1-i}] + (-1)^{i+1}[X,\eta] \mod E^{2m+1}
        \\
        &\equiv -[\ve_{i+1},\ve_{2m+1-i}] + (-1)^{i+1}[X,\eta] \mod E^{2m+1}
    \end{align*}
    It is then easy to show by induction that for each $2\leq i \leq m+1,$
    \begin{align*}
        [\ve_2,\ve_{2m}] &\equiv (-1)^i[\ve_i,\ve_{2m+2-i}] - (i-2)[X,\eta]\mod E^{2m+1}
    \end{align*}
    In particular, for $i=m+1,$ we obtain the result, 
    as $[\ve_{m+1}, \ve_{m+1}]=0$.
\end{proof}
\begin{remark}
Note that for distributions of minimal class $m=1$, so that Lemma \ref{lemma (ii)} becomes trivial, and cannot be used to prove the next lemma. Based on Lemma \ref{lemma: minimal class}, this is the place, where we use the condition that $\rank \, D^3=5$.
\end{remark}
\begin{lemma}
    \label{lemma (iii)}
    \[
        [X,\eta]  \subset E^{2m+1}
    \]
\end{lemma}
\begin{proof}
    We again apply the Jacobi identity to obtain
    \begin{align*}
        [\ve_2,\ve_{2m}] &\equiv \big[[X,\ve_1]\text{ mod } \langle X\rangle,\ve_{2m}\big]
        \\
        & \equiv \big[[X,\ve_{2m}],\ve_1\big] + \big[X,\underbrace{[\ve_1,\ve_{2m}]}_{\eta \text { by \eqref{eta_def}}}\big]
        \\
        & \equiv [X,\eta]\mod E^{2m+1},
    \end{align*}
    where we use in each equivalence the fact that $[X, \ve_{2m}] \subset E^{2m}$, which follows from \eqref{X_stab}. Using Lemma \ref{lemma (ii)}, this then implies that 
    \begin{equation*}
        (1-m)[X,\eta] \equiv [X,\eta]\mod E^{2m+1}.
    \end{equation*}
    Since $m>1$,
    we have that $[X,\eta] \subset E^{2m+1}$.
\end{proof}

\begin{lemma}
    \label{lemma (iv)}
    \[
        [V_{m},E^{2m+1}]\subseteq E^{2m+1}
    \]
\end{lemma}
\begin{proof}
    Recall that $E^{2m+1} = \mathcal{J}^{(m)}\oplus \langle \eta\rangle$; the second involutivity condition \eqref{invol2} of Proposition \ref{prop: involutivity} gives that $[V_{m},\mathcal{J}^{(m)}]\subseteq E^{2m+1}$. Also, applying both involutivity conditions \eqref{invol1}-\eqref{invol2} and the Jacobi identity give
    \begin{align*}
        [V_{m},\eta] &= \big[V_{m},[\ve_1,\ve_{2m}]\big] 
        \\
        &\subseteq \big[[V_{m},\ve_1],\ve_{2m}\big] + \big[\ve_1,[V_{m},\ve_{2m}]\big] 
        \\
        & \subseteq \big[[V_{m},V_{m-1}],\mathcal{J}^{(m)}\big] + \big[\ve_1,[V_{m},\mathcal{J}^{(m)}]\big]
        \\
        &\subseteq [V_{m-1},\mathcal{J}^{(m)}] + [\mathcal{J}_{(m-1)},\mathcal{J}^{(m)}]
        \\
        &= [E,E^{2m}] \subseteq E^{2m+1}.
    \end{align*}
\end{proof}
    
\begin{lemma}
    \label{lemma (vi)}
    \[
        [\ve_1,\eta] \subset E^{2m+1}
    \]
\end{lemma}
\begin{proof}
    Note that by the involutivity condition \eqref{invol1} of Proposition \ref{prop: involutivity} with $i=m-2$,
    \begin{equation}
        [\ve_1,\ve_2]  \subset E^2.
    \end{equation}
    Similarly, by the involutivity condition \eqref{invol2} with $i=m-1$, 
    \begin{equation}
        [\ve_1,\ve_{2m-1}] \subset E^{2m-1}
    \end{equation}
    Therefore, applying Lemma \ref{lemma (i)} then the Jacobi identity yields
    \begin{align}
        [\ve_1,\eta] &\equiv -\big[\ve_1,[\ve_2,\ve_{2m-1}] \text{ mod } E^{2m}\big]
        \\
        & \equiv -\big[[\ve_1,\ve_2],\ve_{2m-1}\big] - \big[\ve_2,[\ve_1,\ve_{2m-1}]\big]
        \\
        &\equiv 0 \mod E^{2m+1}.
    \end{align}
\end{proof}
 
Now, in order to prove the theorem, note that
\begin{align*}
    [E,E^{2m+1}] &= E^{2m+1} + [E,\langle \eta\rangle] 
    \\
    &= E^{2m+1} + [\langle X\rangle,\langle \eta\rangle] + [V_{m},\langle \eta\rangle] + [\langle \ve_1\rangle,\langle \eta\rangle].
\end{align*}
However, the terms on the right-hand side are included in $E^{2m+1}$ by Lemmas \ref{lemma (iii)}, \ref{lemma (iv)}, and \ref{lemma (vi)}, respectively. Therefore, $[E,E^{2m+1}]\subseteq E^{2m+1}$, and the distribution $E^{2m+1}$ is involutive.

Because $E$ is bracket generating, this implies that $E^{2m+1}=T\mathcal{R}_m$. Comparing ranks, we have that
\[
    \text{Rank}(E^{2m+1}) = n+m-1 = \text{Rank}(T\mathcal{R}_m) = 2n-4.
\]
Therefore, $m=n-3$, and $D$ has maximal class at $p$, which completes the proof of Theorem \ref{class_constant} and therefore the first sentence of Theorem \ref{max_class_conjecture}.
\end{proof}
Now prove the second sentence of Theorem \ref{max_class_conjecture}. Let $q_0$ be a point where $D$ is equiregular and assume that $\mu$ is the minimal integer such that $D^\mu(q_0)=T_{q_0}M$. Denote
\begin{equation}
\label{Tanaka_comp}
\mathfrak{g}_{-1}(q_0):=D(q_0), \quad \mathfrak{g}_{-i}(q_0) : =D^{i}(q_0)/D^{i-1}(q_0), \,\, \forall 1\leq j\leq \mu,
\end{equation}
then the graded space
$\mathfrak{m}(q_0)=
\displaystyle{\bigoplus_{j=-\mu}^{-1}\mathfrak{g}_j(q_0)}$,
associated with the filtration \eqref{week_derived}, is endowed with the structure of the graded nilpotent Lie algebra, called the \emph{Tanaka symbol of the distribution $D$ at the point $q_0$}. The \emph{flat distribution $D_{\mathfrak m(q_0)}$ of constant symbol (or type) $\mathfrak m(q_0)$} is defined to be the left-invariant distribution corresponding to the $(-1)$-graded component $\mathfrak g_{-1}(q_0)$
on the simply connected Lie group with Lie algebra $\mathfrak{m}(q_0)$.
Since the flat distribution $D_{\mathfrak m(q_0)}$ is left-invariant, its germs at different points are equivalent \footnote{Note that for non-equiregular points, there is a notion of flat distribution (or nilpotent approximation) as well (\cite{Bellaiche, Jean_book, Ignatovich, Zelenko_review}): it can be seen as a distribution on a homogeneous space, but in contrast to equiregular case there are pairs of points at which the germs of the distribution are not equivalent, so the present arguments do not work.}, and therefore $D_{\mathfrak m(q_0)}$ has constant class. Therefore, by Theorem \ref{class_constant} it is of maximal class at every point.
Then by \cite[Proposition 2.3.2]{Wendell}, the original distribution $D$ at $q_0$ is of maximal class.

\section{Proof of Theorem \ref{corank1_abnormal_thm}}
\label{corank_sec}

We are going to prove the following stronger theorem:
\begin{theorem}
\label{corank1_abnormal_thm_reg}
    Let $D$ be a bracket generating rank $2$ distribution with $5$-dimensional cube on a manifold $M$ and let $D$ have maximal class at $q_0\in M$; i.e., $m(q_0)=n-3$ (the set of such points is generic by Theorem \ref {max_class_conjecture} and contains all points where the distribution is equiregular). Then there exists a regular abnormal extremal trajectory of corank $1$ starting at $q_0$. Moreover, the set of regular abnormal extremal trajectories of corank $1$ starting at $q_0$, considered as unparametrized curves, is not only nonempty but also open and dense—in fact, Zariski open—in the space of regular abnormal extremal trajectories starting at $q_0$, under the identification of this space with the space $\mathbb P\Bigl((D^2)^\perp(q_0)\backslash (D^3)^\perp(q_0)\Bigr)$ (here $(D^\ell)^\perp(q_0):=(D^\ell)^\perp\cap T_{q_0}^*M$).
\end{theorem}\begin{proof}
First, consider a parametrized regular abnormal extremal $\Gamma\colon [0, T]\to \mathcal M$ and let $X$ be a vector field in a neighborhood of $\Gamma$ in $\mathcal M$ that generates the characteristic line distribution $\mathcal C$ and such that $\Gamma$ is an integral curve of $X$. As before, let $\pi:\mathcal M \to M$ be the canonical projection. Set $\gamma:=\pi(\Gamma)$ and let $e^{tX}$ be the flow generated by $X$. Then \cite[Section 4]{Agrachev-Sarychev}, more specifically, relations (4.6) and (4.7) there, imply the following relation between the differential of the endpoint map $F_{\gamma(0), T}$ at $\gamma$ and the lift $\mathcal J$ of the distribution $D$, defined by \eqref{J0}:

\begin{equation}
\label{diff_end_point_jac}
\operatorname{Im} \left( d_\gamma F_{q_0,T} \right)=\mathrm{span}\left\{ d(\pi\circ e^{(T-t)X})_{\Gamma(t)}\mathcal J\left(\Gamma(t)\right): 0\leq t\leq T\right\}\subset d\pi_{\Gamma(T)}H\bigl(\Gamma(T)\bigr).
\end{equation}
The latter inclusion follows from the fact that $X$ generates the Cauchy characteristic distribution $\mathcal C$ of $H$ and it shows why $\gamma$ is an abnormal extremal trajectory in the sense of Definition \ref{abn_def}. 
From relation \eqref {diff_end_point_jac}, the definition of $\mathcal J^{(i)}$ as in \eqref{geod_flag}, and the properties of Lie derivatives, it follows that 
\begin{equation}
\label{class_corank_ineq_1}
d\pi_{\Gamma(T)}\mathcal J^{\bigl(\nu\bigl(\Gamma(T)\bigr)\bigr)}\bigl(\Gamma(T)\bigr)\subset \operatorname{Im} \left( d (F_{q_0,T})_\gamma \right), 
\end{equation}
where $\nu(\lambda)$ is defined by \eqref{abn_ext_class}\footnote{Note that in the real-analytic category the inclusion in \eqref{class_corank_ineq_1} is an equality.}. By construction $\dim\, \mathcal J^{(i)}(\lambda)=n-2+i$ for $0\leq i\leq \nu(\lambda)$ and the fiber of the bundle $\mathcal M$ is $(n-4)$-dimensional we have that $\dim\, d\pi_{\lambda} \mathcal J^{(i)}(\lambda)=i+2$ in this range of $i$. Consequently, \eqref{class_corank_ineq_1} implies that 
$\dim\, \operatorname{Im} \left( d_\gamma F_{q_0,T} \right)\geq \nu\bigl(\Gamma(T)\bigr)+2$. 
Hence,
\begin{equation}
\label{class_corank_ineq_3}
\mathrm{corank}(\gamma)\leq n-2-\nu\bigl(\Gamma(T)\bigr). 
\end{equation}
In particular, if $\nu\bigl(\Gamma(T)\bigr)=n-3$, (i.e., if $\nu$ takes its maximal value), then $\mathrm{corank}(\gamma)\leq 1$. On the other hand, $\gamma$ is an abnormal extremal in the sense of Definition \ref{abn_def}, $\mathrm{corank}(\gamma)>0$, so we conclude that $\mathrm{corank}(\gamma)=1$ in this case. If we assume that $\nu\bigl(\Gamma(0)\bigr) = n-3$, then we can take sufficiently small $T$ so that $\nu\big(\Gamma(T)\bigr)=n-3$, so the corresponding $\gamma$ has corank 1. This proves our theorem, because if $m(q_0) = n - 3$, then the set of points $\lambda$ in the fiber $\pi^{-1}(q_0)$ such that $\nu(\lambda) = n - 3$ is a nonempty Zariski open subset, so, by the above arguments, the projection of a sufficiently small segment of a regular abnormal extremal starting at such a $\lambda$ will have corank 1.
\end{proof}
\begin{remark}
\label{q_0_nonmax_class_rem}
Finally, note that if instead of assuming that the class of the distribution $D$ at $q_0$ is maximal, we assume that there exists a regular abnormal extremal starting at $q_0$ such that the class of the distribution at its endpoint is maximal, then—by the same arguments as in the proof of the previous theorem—the corank of this abnormal extremal trajectory is equal to $1$.
However, we do not know how to exclude the possibility that any regular abnormal extremal trajectory starting at a point where the class of the distribution $D$ is not maximal remains entirely within the locus of points with the same property. Therefore we do not know yet how to remove the assumption on the starting point $q_0$ in Theorems \ref{corank1_abnormal_thm} and \ref{corank1_abnormal_thm_reg}. Perhaps the fact that all points in this locus are points where the distribution $D$ is not equiregular can be used in some way.
\end{remark}

\bibliographystyle{plain}
\bibliography{Bibliography.bib}

\end{document}